\documentclass{amsart}

\usepackage{mathtools}

\usepackage{graphicx}
\newcommand{\apref}[3]{\hyperref[#2]{#1\ref*{#2}#3}}

\usepackage{enumerate}
\usepackage[latin1]{inputenc}
\usepackage{dsfont}
\usepackage{amssymb,amsthm,amsmath}

\input{xy}
\xyoption{all}

\usepackage{mathrsfs}

\theoremstyle{plain}
\newtheorem{prop}{Proposition}[section]
\newtheorem{lemma}[prop]{Lemma}

\newtheorem{thm}[prop]{Theorem}

\newtheorem*{thmAnn}{Theorem A}
\newtheorem*{thmBnn}{Theorem B}

\newtheorem*{corCnn}{Corollary C}

\theoremstyle{definition}

\theoremstyle{remark}
\newtheorem{remark}[prop]{Remark}

\newcommand{\setmid}{\,:\,}

\DeclareMathOperator{\arcosh}{arcosh}
\DeclareMathOperator{\Hyp}{Hyp}

\DeclareMathOperator{\nte}{\mathbf{e}}

\DeclareMathOperator{\End}{End}

\DeclareMathOperator{\Mat}{Mat}
\DeclareMathOperator{\GL}{GL}

\DeclareMathOperator{\SL}{SL}
\DeclareMathOperator{\PSL}{PSL}

\DeclareMathOperator{\tr}{tr}

\DeclareMathOperator{\Ima}{Im}
\DeclareMathOperator{\Rea}{Re}

\newcommand\N{\mathbb{N}}

\newcommand\R{\mathbb{R}}
\newcommand\Z{\mathbb{Z}}
\newcommand\C{\mathbb{C}}

\newcommand{\h}{\mathbb{H}}

\newcommand{\mc}[1]{\mathcal #1}
\newcommand{\mf}[1]{\mathfrak #1}

\newcommand{\wh}{\widehat}

\newcommand{\eps}{\varepsilon}

\DeclareMathOperator{\id}{id}
\DeclareMathOperator{\M}{M}

\newcommand{\sceq}{\coloneqq}

\newcommand{\bmat}[4]{\begin{bmatrix} #1&#2\\#3&#4\end{bmatrix}}
\newcommand{\textmat}[4]{\left(\begin{smallmatrix} #1&#2 \\ #3&#4
\end{smallmatrix}\right)}
\newcommand{\textbmat}[4]{\left[\begin{smallmatrix} #1&#2 \\ #3&#4
\end{smallmatrix}\right]}

\usepackage{hyperref}

\begin{document}

\title{Eisenstein series twisted with non-expanding cusp monodromies}
\author[K.\@ Fedosova]{Ksenia Fedosova}
\address{KF: Albert-Ludwigs-Universit\"at Freiburg, Mathematisches Institut, Ernst-Zermelo-Str. 1, 79104 Freiburg im Breisgau, Germany}
\email{ksenia.fedosova@math.uni-freiburg.de}
\author[A.\@ Pohl]{Anke Pohl}
\address{AP: University of Bremen, Department 3 -- Mathematics, Bibliothekstr.\@ 
5,  28359 Bremen, Germany}
\email{apohl@uni-bremen.de}
\begin{abstract} 
Let $\Gamma$ be a geometrically finite Fuchsian group and suppose that $\chi\colon\Gamma\to\nobreak\GL(V)$ is a finite-dimensional representation with non-expanding cusp monodromy. We show that the parabolic Eisenstein series for $\Gamma$ with twist $\chi$ converges on some half-plane. Further, we develop Fourier-type expansions for these Eisenstein series. 
\end{abstract}
\maketitle

\section{Introduction}

Recently, interest emerged in developing Selberg-type trace formulas with \textit{non-unitary} twists for locally symmetric spaces and beyond. 

M\"uller \cite{Mueller_STF} established a Selberg trace formula with non-unitary twists for \textit{compact} locally symmetric spaces $\Gamma\backslash G/K$. We refer to \cite{Spilioti2015, Fedosova_nonunitary} for modified trace formulas and closely related Selberg zeta functions for certain spaces of rank $1$. Subsequently, Deitmar--Monheim \cite{Deitmar_Monheim_traceformula} and Deitmar \cite{Deitmar_locally_compact} provided Selberg-type trace formulas for certain \textit{compact} quotient spaces $\Gamma\backslash G$.

Motivated by the quest for a (still to be developed) Selberg-type trace formula with \textit{non-unitary} twists for \textit{noncompact} spaces, Deitmar and Monheim \cite{Deitmar_Monheim_eisenstein} (see also~\cite{Monheim}) specialized to $G=\PSL_2(\R)$ and investigated twisted (parabolic, non-holomorphic) Eisenstein series for nonuniform, cofinite Fuchsian groups $\Gamma$. 

To survey these results in more detail, let $\h$ denote the hyperbolic plane, let $\Gamma$ be a geometrically finite Fuchsian group and let $\chi\colon \Gamma\to\nobreak \GL(V)$ be a representation of $\Gamma$ on a finite-dimensional unitary space $V$. We refer to Section~\ref{sec:prelims} below for definitions of these notions.

Deitmar and Monheim  studied the $\chi$-twisted Eisenstein series $E_\chi$ for the case that~$\Gamma$  is \textit{cofinite} and $\chi$ is \textit{unitary at cusps}, i.\,e., for any parabolic element $p\in\Gamma$ the endomorphism $\chi(p)$ is unitary. In this case, the (parabolic, non-holomorphic) Eisenstein series~$E_{\mf c,\chi}$ at the cusp (or rather cusp representative) $\mf c$ is given by (initially only formally)
\[
 E_{\mf c,\chi}(z,s)\sceq \sum_{[g]\in\Gamma_{\mf c}\backslash\Gamma} \big(\Ima(\sigma_{\mf c}^{-1}g.z)\big)^s \chi(g^{-1})\circ P_{\mf c}   \colon V\to V
\]
where $z\in\h$, $s\in\C$ (or rather an appropriate subset of $\C$), $\Gamma_{\mf c}$ is the parabolic stabilizer group of $\mf c$ in $\Gamma$, $\sigma_{\mf c}\in\PSL_2(\R)$ is a scaling matrix at $\mf c$, and $P_{\mf c}\colon V\to V$ is the orthogonal projection onto the space $V_{\mf c}$ of elements fixed by $\chi(\Gamma_{\mf c})$. The $\chi$-twisted Eisenstein series $E_\chi$ is then (again initially only formally) given by
\begin{equation}\label{Eisenintro}
 E_\chi \sceq \sum_{\mf c\in \mf C} E_{\mf c,\chi}
\end{equation}
where $\mf C$ is a (complete, $\Gamma$-reduced) set of representatives for the cusps of $\Gamma$.

For example, for $\PSL_2(\Z)$, using the cusp representative~$\mf c=\infty$, the twisted Eisenstein series is 
\[
 E_\chi(z,s) = E_{\infty,\chi}(z,s) = \frac12\sum_{\substack{m,n\in\Z\\ (m,n)=1}} \frac{\big(\Ima z\big)^s}{|mz+n|^{2s}}\cdot \chi\left( \bmat{n}{*}{-m}{*}\right)\circ P_\infty
\]
where $\textbmat{n}{*}{-m}{*}$ is any element of $\PSL_2(\Z)$ with first column as indicated, and $P_\infty$ is the orthogonal projection of $V$ onto 
\[
V_\infty = \ker\left( \id_V - \chi\left(\bmat{1}{1}{0}{1}\right)\right).
\]

Deitmar and Monheim \cite{Deitmar_Monheim_eisenstein} showed that there exists $\alpha >0$ such that the Eisenstein series $E_\chi$ (for any cofinite Fuchsian group~$\Gamma$ and any twist~$\chi$ that is unitary at cusps) converges absolutely and locally uniformly for $\Rea s > 1+\alpha$. Moreover, they investigated the existence of a meromorphic continuation of $E_\chi$. 

In this article we discuss the convergence of Eisenstein series $E_\chi$ for any geometrically finite Fuchsian group~$\Gamma$ (also allowed to be non-cofinite) with twists $\chi$ having \textit{non-expanding cusp monodromy}, i.\,e., for every parabolic element $p\in\Gamma$, all eigenvalues of the endomorphism $\chi(p)$ are on the unit circle. We provide an upper bound for the abscissa of convergence of $E_\chi$. In addition, we present Fourier-type expansions of these Eisenstein series.

For any Fuchsian group $\Gamma$ with parabolic elements, the class of representations with non-expanding cusp monodromy is a proper superset of those which are only unitary at cusps. We refer to \cite[Section~5]{FP_szf} for examples of representations with non-expanding cusp monodromy that are not unitary at cusps. Throughout the last few decades, representations with non-expanding cusp monodromy have already been seen to be of huge interest. We provide a few examples in what follows.

Knopp and Mason investigated generalized vector-valued modular forms for the modular group~$\PSL_2(\Z)$ twisted with representations having non-expanding cusp monodromy, see e.\,g.\@ \cite{Knopp_Mason_definition, Knopp20032, Knopp_Mason_illinois, Knopp2011, Knopp_Mason2012}. One motivation for their work is the occurence of such forms in certain conformal field theories. Additional results in this direction were obtained, e.\,g., by Daughton \cite{Daughton2016} as well as by Kohnen and Mason \cite{Kohnen_Mason, Kohnen_Martin, Kohnen, Kohnen_Mason2}. A nice survey article, albeit with a different focus, is provided by Franc and Mason \cite{Cameron2016}.

Moreover, representations with non-expanding cusp monodromy (for arbitrary cofinite Fuchsian groups) play a crucial role in \cite{EKMZ} (who also invented the name) in their study of Lyapunov exponents of flat bundles on curves. 

In \cite{FP_szf}, we established, for any geometrically finite Fuchsian group $\Gamma$, convergence of the Selberg zeta functions for $\Gamma$ with twists of non-expanding cusp monodromy. We further showed that for a certain rather large subclass of Fuchsian groups these Selberg zeta functions have a meromorphic continuation to all of~$\C$. Moreover, we proved that for twists without non-expanding cusp monodromy, the twisted Euler products used for the definition of Selberg zeta functions do not converge. This indicates that also for Eisenstein series the representations with non-expanding cusp monodromy form a natural (in a certain sense maximal) family of twists. The results we present here are a natural continuation of our previous work and they are complementary to \cite{Knopp_Mason_definition, Knopp20032, Knopp_Mason_illinois, Kohnen_Mason, Knopp2011, Knopp_Mason2012,Daughton2016}.

The first main result of this article is a proof that the Eisenstein series $E_\chi$ with twists $\chi$ of non-expanding cusp monodromy converge in certain half-spaces in the $s$-variable. 

\begin{thmAnn}
Suppose that $\chi$ has non-expanding cusp monodromy. Then there exists $C>0$ such that the Eisenstein series $E_{\chi}$ converges absolutely and compactly in $\h\times \{s\in\C \mid \Rea s > 1+C\}$. 
\end{thmAnn}

As explained in \cite[Introduction]{Deitmar_Monheim_eisenstein}, a key ingredient for establishing convergence is to show that, for any cusp representative $\mf c$, the operator $\chi(g^{-1})\circ P_{\mf c}$ can be bounded polynomially in terms of $\Ima(\sigma_{\mf c}^{-1}g.z)$ with all bounding constants depending continuously on $z\in\h$ and being independent of (almost all) $[g]\in\Gamma_{\mf c}\backslash\Gamma$. For providing such a result we take advantage of our detailed study in \cite{FP_szf} on the weights that $\chi$ attributes to periodic geodesics  with long cusp excursions on $\Gamma\backslash\h$. The findings allow us now to establish the following polynomial bound.

\begin{thmBnn}
For every cusp representative $\mf c$ there exist 
\begin{enumerate}[{\rm (i)}]
\item $C_{\mf c}>0$, 
\item a continuous map $b\colon \h\to \R$, and 
\item for all $z\in\h$  a finite set $H_e(z)\subseteq \Gamma_{\mf c}\backslash\Gamma$ such that for every compact set $K\subseteq\h$, the set $\bigcup_{z\in K} H_e(z)$ is finite
\end{enumerate}
such that for all $z\in\h$ and all $[g]\in\Gamma_{\mf c}\backslash\Gamma$, $[g]\notin H_e(z)$, we have
\[
 \| \chi(g^{-1})\circ P_{\mf c} \| \leq b(z) \big(\Ima(\sigma_{\mf c}^{-1}.z)\big)^{-\eps\cdot C_{\mf c}} \big( \Ima(\sigma_{\mf c}^{-1}g.z) \big)^{\eps\cdot C_{\mf c}}
\]
where
\[
 \eps \sceq 
 \begin{cases}
  1 & \text{if $\ \Ima(\sigma_{\mf c}^{-1}g.z) \geq \Ima(\sigma_{\mf c}^{-1}.z)$,}
  \\
  -1 & \text{if $\ \Ima(\sigma_{\mf c}^{-1}g.z) < \Ima(\sigma_{\mf c}^{-1}.z)$.}
 \end{cases}
\]
\end{thmBnn}

In addition we show that---as in \cite{Deitmar_Monheim_eisenstein}---the constant $C$ in Theorem~A is governed by the order of the growth in Theorem~B.

\begin{corCnn}
Let $\mf C$ be a set of representatives (as in \eqref{Eisenintro}) for the cusps of $\Gamma$. For $\mf c\in\mf C$ let $C_{\mf c}>0$ be as in Theorem~B. Then the constant $C>0$ in Theorem~A can be chosen as
\[
 C\sceq \max_{\mf c\in\mf C} C_{\mf c}.
\]
\end{corCnn}

The next task in this line of research is to study the possibility of meromorphic continuations of the $\chi$-twisted Eisenstein series. For this, an advance of the spectral theory in such $\chi$-twisted situations is necessary.

Our second main result is an intermediate step in this direction. We provide a Fourier-type expansion of $\chi$-twisted Eisenstein series. As soon as the representation~$\chi$ is not unitary in some cusp $\mf c$, the twisted Eisenstein series is not periodic along the horocycles centered at $\mf c$. More precisely, there are directions in the vector space $V$ projected to which the $\chi$-twisted Eisenstein series is not periodic along the horocycles centered at $\mf c$. Hence, the Eisenstein series does not allow for a Fourier expansion in the usual sense. 

However, a periodization of the Eisenstein series along these horocycles is possible, and the periodized function has a Fourier expansion. Using this Fourier expansion and then inverting the periodization process provides a Fourier-type expansion of the Eisenstein series. If $\chi$ is unitary or unitary in cusps, then this Fourier-type expansion coincides with the known Fourier-type expansion, see \cite[Theorem~3.1.2, 3.1.3]{Venkov_book} and \cite[Theorem~3.6]{Deitmar_Monheim_eisenstein}. We refer to Theorem~\ref{thm:expansion} below for the precise statement of the Fourier-type expansion.

In Section~\ref{sec:prelims} below we provide the necessary background material. In Section~\ref{sec:thmB} below we discuss the proof of Theorem~B. The proofs of Theorem~A and Corollary~C, which both depend on Theorem~B, are the content of Section~\ref{sec:thmA} below. In Section~\ref{sec:expansion} below we present the Fourier-type expansion of twisted Eisenstein series.

\subsubsection*{Acknowledgement} The authors are grateful to Julie Rowlett for generously providing Lemma~\ref{lem:julie}. They wish to thank the Max Planck Institute for Mathematics in Bonn for great hospitality and excellent working conditions during part of the preparation of this manuscript. Further, AP acknowledges support by the DFG grant PO 1483/2-1.

\section{Preliminaries and notation}\label{sec:prelims}

\subsection{Elements of hyperbolic geometry}
For the hyperbolic plane~$\h$ we use throughout the upper half-plane model
\[
 \h \sceq \{ z\in\C\mid \Ima z > 0\},\qquad ds^2_z = \frac{dzd\overline{z}}{(\Ima z)^2}. 
\]
We identify its group of orientation-preserving Riemannian isometries with 
\[
G\sceq\PSL_2(\R) = \SL_2(\R)/\{\pm\id\}
\]
in the usual way, as recalled in what follows. The action of $G$ on $\h$ is then given by the well-known fractional linear transformations.

We denote the elements of $G$ by any of their representing matrices but in square brackets. Thus,  if $g\in G$ is represented by $\textmat{a}{b}{c}{d}\in\SL_2(\R)$, then we write
\[
 g = \bmat{a}{b}{c}{d}.
\]
The action of $g$ on $\h$ is then
\[
 g.z \sceq \frac{az+b}{cz+d}\qquad (z\in\h).
\]
We denote by $d_\h$ the metric on $\h$ which is induced by the Riemannian metric of~$\h$. We identify the geodesic boundary $\partial\h$ in the usual way with $\wh\R\sceq\R\cup\{\infty\}$. As is well-known, the action of $G$ extends continuously to $\wh\R$. 

Let $\Gamma$ be a Fuchsian group, i.\,e., a discrete subgroup of~$G$. We say that~$\Gamma$ is \emph{geometrically finite} if it is finitely generated. The group~$\Gamma$ is geometrically finite if and only if the orbit space $\Gamma\backslash\h$ has only finitely many ends, or equivalently, if for any compact subset $K\subseteq \Gamma\backslash\h$, the set $\big(\Gamma\backslash\h\big)\smallsetminus K$ has only finitely many connected components. The group~$\Gamma$ is called \emph{uniform} or \emph{cocompact} if $\Gamma\backslash\h$ is compact. It is called \emph{cofinite} if $\Gamma\backslash\h$ has finite volume (with respect to the volume form induced from~$\h$). Finally, $\Gamma$ is called \emph{elementary} if $\Gamma$ fixes a point in $\h\cup\partial\h$ or a geodesic on $\h$. 

Let $g\in G\smallsetminus\{\id\}$. We call $g$ \textit{parabolic} if it fixes exactly one point in $\partial\h$, we call it \textit{hyperbolic} if it fixes exactly two points in $\partial\h$. Further, we call a point $\mf c\in \wh\R$ \textit{cuspidal} or a \textit{cusp representative} if it is the fixed point of some parabolic element $p\in\Gamma$. A \textit{cusp} of $\Gamma$ is then the $\Gamma$-orbit of a cuspidal point, or, more precisely, an equivalence class of cusp representatives where equivalence is induced by the action of $\Gamma$ on $\wh\R$. If $\Gamma$ has a cusp, then $\Gamma$ is nonuniform. (The converse is not true within the set of geometrically finite Fuchsian groups, but within the set of cofinite Fuchsian groups.)

For every cusp representative $\mf c$ we let
\[
 \Gamma_{\mf c} \sceq \{ g\in\Gamma\mid g.\mf c=\mf c\}
\]
denote its stabilizer group in $\Gamma$. We denote the elements of the coset space $\Gamma_{\mf c}\backslash\Gamma$ by $[g]$ with $g\in\Gamma$ being any representative of this equivalence class. 

From now on we shall assume that $\Gamma$ is geometrically finite. Our results are valid for all geometrically finite Fuchsian groups; they are nontrivial if $\Gamma$ is non-elementary and has at least one cusp.

\subsection{Representations} 
Let $\chi\colon \Gamma\to\GL(V)$ be a representation of $\Gamma$ on a finite-dimensional unitary space $V$. We suppose throughout that $\chi$ has \emph{non-expanding cusp monodromy}, that is, for every parabolic element $p\in\Gamma$, all eigenvalues of the endomorphism $\chi(p)$ have modulus $1$. 

Let $\End(V)$ denote the vector space of endomorphisms on $V$. We write $AB$ for the composition of  $A,B\in \End(V)$, thus $AB=A\circ B$. Throughout we fix a norm $\|\cdot\|$ on $\End(V)$. Since, e.\,g., the operator norm as well as the trace norm on $\End(V)$ are submultiplicative and all norms on $\End(V)$ are equivalent (due to finite-dimensionality of $\End(V)$), the norm $\|\cdot\|$ is essentially submultiplicative. Thus, for any choice of $\|\cdot\|$ there exists $C>0$ such that for all $A,B\in\End(V)$ we have
\[
 \|AB\| \leq C \| A\| \|B\|.
\]

\subsection{Eisenstein series}

Let $\mf c$ be a cusp representative. Let 
\[
 V_{\mf c} \sceq \{ v\in  V\mid \forall\, p\in\Gamma_{\mf c}\colon \chi(p)v=v\}
\]
be the subspace of $V$ that is stabilized elementwise by $\chi(\Gamma_{\mf c})$, and let 
\[
 P_{\mf c}\colon V\to V
\]
denote the orthogonal projection of $V$ onto $V_{\mf c}$. Choose $\sigma_{\mf c}\in G$ such that $\sigma_{\mf c}^{-1}.\mf c=\infty$ and 
\[
 \sigma_{\mf c}^{-1} \Gamma_{\mf c} \sigma_{\mf c} = \bmat{1}{\Z}{0}{1}.
\]
Thus, $\sigma_{\mf c}^{-1}$ transforms the cusp representative $\mf c$ into $\infty$, and normalizes the cusp width at $\mf c$ to $1$.

We define the (parabolic, non-holomorphic) \textit{Eisenstein series at $\mf c$ with twist $\chi$} formally by 
\begin{equation}\label{defEisencusp}
 E_{\mf c,\chi}(z,s)\sceq \sum_{g\in\Gamma_{\mf c}\backslash\Gamma} \left(\Ima\big( \sigma_{\mf c}^{-1}g.z\big)\right)^s \chi(g^{-1})P_{\mf c}\colon V\to V
\end{equation}
where $z\in\h$, $s\in\C$. Determining an upper bound on the abscissa of convergence within $\h\times\C$ is one goal of this article.

One easily checks that the endomorphism $\chi(g^{-1})P_{\mf c}$ in \eqref{defEisencusp} does not depend on the choice of the representative $g$ of $[g]$. Furthermore, \eqref{defEisencusp} does not depend on the choice of $\sigma_{\mf c}$. Thus, \eqref{defEisencusp} is indeed well-defined in its domain of convergence.

Let $\mf C$ be a (reduced) set of representatives for the cusps of $\Gamma$. This means that for each cusp of~$\Gamma$, the set $\mf C$ contains \emph{exactly one} representative.  Since $\Gamma$ is geometrically finite, $\mf C$ is finite. 

The (parabolic, non-holomorphic) \textit{Eisenstein series with twist $\chi$} is then formally given by
\begin{equation}\label{defEisen}
 E_\chi \sceq \sum_{\mf c\in\mf C} E_{\mf c,\chi}.
\end{equation}
Again, one easily checks that \eqref{defEisen} does not depend on the choice of $\mf C$. Since we will consider only parabolic non-holomorphic Eisenstein series (and no holomorphic or hyperbolic or elliptic ones), we refer to them only as `Eisenstein series.'

The definitions \eqref{defEisencusp} and \eqref{defEisen} are generalizations of the definitions of Eisenstein series without twists, with unitary twists as well as with twists that are unitary in the cusps.

If $\Gamma$ has no cusp, then the set~$\mf C$ in~\eqref{defEisen} is empty and hence the twisted Eisenstein series is constant zero. If $\Gamma$ is elementary, then the sum in \eqref{defEisencusp} contains at most one non-vanishing summand, and hence there are no convergence issues. In what follows we tacitly assume that $\Gamma$ is non-elementary and has at least one cusp.

\subsection{Further notation}

We use $\Gamma\smallsetminus \Lambda$ to denote the set of elements in $\Gamma$ that are not in $\Lambda$ (`set-difference'). Note that coset spaces are written as $\Lambda\backslash\Gamma$. 

Let $M$ be a set. For any maps $f_1\colon M\to\End(V)$, $f_2\colon M\to\R$, $g\colon M\to\R$ we write $f_j\ll g$ ($j=1,2$) if there exists $C>0$ such that 
\[
 \forall\, m\in\M\colon \| f_1(m) \| \leq C g(m)
\]
and
\[
 \forall\, m\in\M\colon |f_2(m)| \leq C g(m),
\]
respectively.

\section{Proof of Theorem~B}\label{sec:thmB}

Without loss of generality we suppose throughout that $\infty$ is a cusp representative for $\Gamma$, and that we can choose $\sigma_\infty = \textbmat{1}{0}{0}{1}$. For a proof of Theorem~B it then obviously suffices to establish the following theorem.

\begin{thm}\label{thm:estimate}
There exists a constant $c_1>0$ such that for all $z\in\h$ there exists $c_2(z)>0$ and a finite set $H_e(z)\subseteq \Gamma_\infty\backslash\Gamma$ such that for all $[g] \in \Gamma_\infty\backslash\Gamma$, $[g]\notin H_e(z)$, we have
\begin{equation}\label{eq:estimate}
 \| \chi(g^{-1})P_\infty\| \leq c_2(z) \cdot \big( \Ima z \big)^{-\eps\cdot c_1} \big( \Ima(g.z) \big)^{\eps\cdot c_1}
\end{equation}
where
\[
 \eps \sceq 
 \begin{cases}
  1 & \text{if $\ \Ima(g.z) \geq \Ima z$,}
  \\
  -1 & \text{if $\ \Ima(g.z) < \Ima z$.}
 \end{cases}
\]
The map $z\mapsto c_2(z)$ can be chosen to be continuous on $\h$, the map $z\mapsto H_e(z)$ can be chosen to be locally finite.
\end{thm}

As a preparation for the proof of Theorem~\ref{thm:estimate} we first provide---in Proposition~\ref{prop:hyp} below---a bound on $\|\chi(g^{-1})P_\infty\|$ in terms of the displacements $d_\h(z,h.z)$ for $z\in\h$ and $h\in\Gamma$ being a \textit{hyperbolic} representative of $[g]\in\Gamma_\infty\backslash\Gamma$. Lemma~\ref{enoughhyp} below shows that indeed all but finitely many representatives of $[g]$ are hyperbolic. In Proposition~\ref{prop:frac} below we then study the relation between $d_\h(z,h.z)$ and $\Ima(g.z)$.

Throughout we set 
\[
 p_k\sceq \bmat{1}{k}{0}{1} \qquad (k\in\Z).
\]
Since $\sigma_\infty =\textbmat{1}{0}{0}{1}$, the cusp width at $\infty$ is $1$ and 
\[
 \Gamma_\infty = \{ p_k \mid k\in\Z\}.
\]
Let 
\[
 c(\infty) \sceq \inf \left\{ |c| \ \left\vert\  \exists\, g\in\Gamma\smallsetminus\Gamma_\infty  \colon g=\bmat{\ast}{\ast}{c}{\ast} \right.\right\}.
\]
As is well-known, the discreteness of $\Gamma$ and $\infty$ being a cuspidal point of $\Gamma$ of width~$1$ imply that $c(\infty)\geq 1$.

In the proof of the following lemma, for any $h\in G$, we use $|\tr(h)|$ to denote the absolute value of the trace of a representative of $h$ in $\SL_2(\R)$. Note that $|\tr(h)|$ is indeed independent of the choice of the representative.

\begin{lemma}\label{enoughhyp}
Let $[g]\in \Gamma_\infty\backslash\Gamma$, $[g]\not=\Gamma_\infty$. Let $g$ be a representative of $[g]$ and set 
\[
 E\sceq \{ k\in \Z \mid \text{$p_kg$ is not hyperbolic}\}.
\]
Then 
\[
 \max E - \min E \leq 2 \left\lfloor \frac{4}{c(\infty)}\right\rfloor + 1.
\]
In particular, if $L\subseteq\Z$ is a subset which contains at least 
\[
 2\left\lfloor \frac{4}{c(\infty)}\right\rfloor + 2
\]
consecutive integers, then there exists $k\in L$ such that $p_kg$ is hyperbolic.
\end{lemma}

\begin{proof}
If all representatives of $[g]$ are hyperbolic, then nothing remains to be shown. Thus, suppose that $[g]$ has a non-hyperbolic representative, and let 
\[
 g = \bmat{a}{b}{c}{d}
\]
be such a representative. We investigate which other representatives of $[g]$ might be non-hyperbolic. Clearly, $[g] = \{p_kg\mid k\in\Z\}$. Suppose that $k\in\Z$ is such that $p_kg$ is not hyperbolic. Thus,
\[
 2 \geq |\tr(p_kg)| = |a+d + kc| \geq k|c| - |a+d| = |kc| - |\tr g|.
\]
Since $|\tr g|\leq 2$ and $c\not=0$ (because $[g]\not=\Gamma_\infty$),
\[
 |k|\leq \frac{4}{|c|} \leq \frac{4}{c(\infty)}. 
\]
This completes the proof.
\end{proof}

For $[g]\in\Gamma_\infty\backslash\Gamma$, $[g]\not=\Gamma_\infty$, let $\Hyp([g])$ denote the set of hyperbolic representatives of $[g]$. 

\begin{prop}\label{prop:hyp}
There exist a continuous map $c_3\colon\h\to\R$ and a constant $c_4>0$ such that for each $[g]\in \Gamma_\infty\backslash\Gamma$, $[g]\not=\Gamma_\infty$, we have
\begin{equation}\label{ineq}
 \|\chi(g^{-1})P_\infty\| \leq c_3(z)\cdot \inf_{h\in\Hyp([g])} e^{c_4 d_\h (z, h.z) }.
\end{equation}
\end{prop}

\begin{proof}
Let $[g]\in \Gamma_\infty\backslash\Gamma$, $[g]\not=\Gamma_\infty$. By Lemma~\ref{enoughhyp},  $\Hyp([g])\not=\emptyset$.  Thus, the right hand side of \eqref{ineq} is well-defined. Pick $h\in \Hyp([g])$. Then
\[
 \|\chi(h^{-1})P_\infty\| \ll \| \chi(h^{-1})\| \cdot \| P_\infty\| \ll \| \chi(h^{-1})\|.
\]
Since $\chi$ has non-expanding cusp monodromy, \cite[Corollary~3.4]{FP_szf} shows that there exists a continuous map $c'_3\colon\h\to\R$ and a constant $c_4>0$ (both independent of $[g]$ and $h$) such that
\[
 \| \chi(h^{-1})\| \leq c'_3(z) e^{c_4 d_\h (z, h.z)}.
\]
This completes the proof.
\end{proof}

\begin{prop}\label{prop:frac}
For every $z\in\h$ there exist a finite set $H_e(z)\subseteq\Gamma_\infty\backslash\Gamma$ and $c_5(z)>0$ such that $\Gamma_\infty\in H_e(z)$ and for all $[g]\in\Gamma_\infty\backslash\Gamma$ with $[g]\notin H_e(z)$, we have
\begin{equation}\label{distest}
-\infty < \inf_{h\in \Hyp([g])} d_\h(h.z, z) \leq c_5(z)+\left| \log \frac{\Ima(g.z)}{\Ima z} \right|.
\end{equation}
The map $z\mapsto c_5(z)$ can be chosen to be continuous on $\h$. The map $z\mapsto H_e(z)$ can be chosen such that for every compact set $K\subseteq \h$, the set $\bigcup_{z\in K} H_e(z)$ is finite (the \emph{finiteness condition}).
\end{prop}

\begin{proof}
Let 
\[
 \tau\sceq 2\left\lfloor \frac{4}{c(\infty)} \right\rfloor + 2
\]
and let $c_6\colon\h\to\R$ be a continuous map such that $c_6(z)\in (0, \Ima z)$ for all $z\in\h$. Depending on $\tau$ and $c_6$ define the continuous maps $c_5, c_7\colon \h\to\R$ by 
\[
 c_7(z) \sceq 10 \sqrt{\frac{\tau^2}{c_6(z)^2} + 1}
\]
and
\[
 c_5(z) \sceq 2 \log c_7(z).
\]
Let $K\subseteq\h$ be a compact nonempty subset, and choose a compact set $K'\subseteq\h$ such that for all $z\in K$
\[
 \big\{ w\in\h \setmid |\Rea w - \Rea z |< \tau,\  |\Ima w -\Ima z| < c_6(z) \big\} \subseteq K'.
\]
Let 
\[
 \Lambda \sceq \{ h\in\Gamma\mid h.K\cap K'\not=\emptyset\}.
\]
Since $\Gamma$, as a Fuchsian group, acts properly discontinuously on $\h$, the set $\Lambda$ is finite. Let 
\[
 P \sceq \{ [g]\in\Gamma_\infty\backslash\Gamma\mid g\in\Lambda\}.
\]
Let $z = x_0+iy_0\in K$ ($x_0,y_0\in\R$) and $[g]\in\Gamma_\infty\backslash\Gamma$, $[g]\notin P$. Pick a representative~$g$ of $[g]$. Since for each $k\in\Z$, $p_kg$ is a representative of $[g]$ as well and 
\[
 p_kg.z = g.z + k,
\]
Lemma~\ref{enoughhyp} implies that we find at least one hyperbolic representative $h$ of $[g]$ such that
\[
 |\Rea(h.z) - x_0| < \tau.
\]
Since $[g]\notin P$, 
\[
 |\Ima(h.z) - y_0| \geq c_6(z).
\]
Set
\[
 x_1 \sceq \Rea(h.z)\quad\text{and}\quad y_1\sceq \Ima(h.z).
\]
Then (see \cite[p. 80]{Iversen92})
\begin{align*}
d_\h(h.z,z) & = d_\h(x_1+iy_1, x_0+iy_0)
\\
& = \arcosh \left( 1 + \frac{(x_1- x_0)^2 + (y_1 - y_0)^2}{2 y_0  y_1}\right)
\\
& = 2 \log \frac{\sqrt{(x_1-x_0)^2 + (y_1-y_0)^2} + \sqrt{(x_1-x_0)^2 + (y_1+y_0)^2}}{2 \sqrt{y_1y_0}}
\\
& \leq 2 \log \frac{ \sqrt{\tau^2 + (y_1-y_0)^2} + \sqrt{\tau^2 + (y_1+y_0)^2} }{2\sqrt{y_1y_0}}
\\
& \leq 2 \log \frac{ c_7(z) \sqrt{(y_1-y_0)^2} + c_7(z) \sqrt{(y_0+y_1)^2}}{ 2\sqrt{y_1y_0}}
\\
& = c_5(z) +  \left| \log\frac{y_1}{y_0} \right|.
\end{align*}
This completes the proof.
\end{proof}

\begin{proof}[Proof of Theorem~\ref{thm:estimate}]
In Proposition~\ref{prop:frac} pick a continuous map $z\mapsto c_5(z)$ and a map $z\mapsto H_e(z)$ that satisfies the finiteness condition. Let $z\in\h$ and $[g]\in\Gamma_\infty\backslash\Gamma$, $[g]\notin H_e(z)$. Applying Propositions~\ref{prop:hyp} and \ref{prop:frac} (in this order) and using the notation established there we find
\begin{align*}
 \| \chi(g^{-1}) P_\infty\| & \leq c_3(z) \cdot \inf_{h\in\Hyp([g])} e^{c_4 d_\h (h.z, z)}
 \\
  & \leq c_3(z) \cdot \exp\left( c_4\cdot c_5(z) + c_4 \left| \log\frac{\Ima(g.z)}{\Ima(z)}\right| \right).
\end{align*}
Setting $c_2(z)\sceq c_3(z) e^{c_4c_5(z)}$ and $c_1\sceq c_4$ completes the proof. 
\end{proof}

\section{Proofs of Theorem~A and Corollary~C}\label{sec:thmA}

We continue to assume that $\infty$ is a cusp representative of $\Gamma$, and that the cusp width at $\infty$ is $1$. For the proofs of Theorem~A and Corollary~C, it obviously suffices to show the claimed convergence and bounds for the Eisenstein series at the cusp~$\infty$ as stated in the following theorem. 

\begin{thm}\label{thm:conv}
Let $c_1$ be as in Theorem~\ref{thm:estimate}.  Then the Eisenstein series $E_{\infty,\chi}$ converges absolutely and compactly in $\h\times \{s\in\C\mid\Rea s > 1+c_1\}$. 
\end{thm}

\begin{proof}
For $z\in\h$ choose $H_e(z)$ as in Theorem~\ref{thm:estimate} and set 
\begin{align*}
 H_+(z) &\sceq \big\{ [h] \in \Gamma_\infty\backslash\Gamma \ \big\vert\  [h]\notin H_e(z),\ \Ima(h.z) \geq \Ima z \big\},
 \\
 H_-(z) & \sceq \big\{ [h] \in \Gamma_\infty\backslash\Gamma \ \big\vert\  [h]\notin H_e(z),\ \Ima(h.z) < \Ima z \big\}.
\end{align*}
For $z\in\h$ and $s\in\C$ we (formally) have
\begin{align*}
 E_{\infty,\chi}(z;s) & = \sum_{[g]\in H_e(z)} \big(\Ima(g.z)\big)^s \chi(g^{-1})P_\infty
 \\
  & \ \quad + \sum_{[g]\in H_+(z)} \big(\Ima(g.z)\big)^s \chi(g^{-1})P_\infty + \sum_{[g]\in H_-(z)} \big(\Ima(g.z)\big)^s \chi(g^{-1})P_\infty.
\end{align*}
Theorem~\ref{thm:estimate} shows
\begin{equation}\label{bound1}
 \sum_{[g]\in H_+(z)} \big(\Ima(g.z)\big)^s \chi(g^{-1})P_\infty \ll_z \sum_{[g]\in H_+(z)} \big|\Ima(g.z)\big|^{\Rea s + c_1}
\end{equation}
and
\begin{equation}\label{bound2}
 \sum_{[g]\in H_-(z)} \big(\Ima(g.z)\big)^s \chi(g^{-1})P_\infty \ll_z \sum_{[g]\in H_-(z)} \big|\Ima(g.z)\big|^{\Rea s - c_1}.
\end{equation}
The implied constants depend continuously on $z$. Moreover, the set $H_e(z)$ is finite. 

By Theorem~\ref{thm:estimate}, for any given compact set $K\subseteq \h$, the maps $z\mapsto H_e(z)$ and $z\mapsto H_{\pm}(z)$ can be chosen to be constant on $K$. Therefore, the right hand side of \eqref{bound1} and \eqref{bound2} converges absolutely and uniformly in $K_z\times K_s$ for every compact subsets $K_z\subseteq\h$ and 
\[
 \begin{cases}
  K_s\subseteq\{s \in \C \mid \Rea s>1-c_1\} & \text{for~\eqref{bound1},}
  \\[2mm]
  K_s\subseteq\{s \in \C \mid \Rea s>1+c_1\} & \text{for~\eqref{bound2}.}
 \end{cases}
\]
Thus, the Eisenstein series $E_{\infty,\chi}$ converges absolutely and compactly in 
\[
 \h \times \{s\in\C\mid \Rea s > 1+c_1\}.
\]
This completes the proof.
\end{proof}

\section{Fourier-type expansion of Eisenstein series}\label{sec:expansion}

If the representation $\chi$ is not unitary in some cusp, then the $\chi$-twisted Eisenstein series $E_\chi$ is not periodic along the horocycles centered at this cusp. Thus, it does not have a Fourier expansion in the usual sense.

However, for any pair $\mf a,\mf b\in\mf C$ of cusp representatives we can periodize $E_{\mf a,\chi}$ along the horocyclic direction of the cusp represented by $\mf b$. This periodization allows us to develop a Fourier-type expansion of $E_\chi$.

More precisely, we first note (see Remark~\ref{rem:periodic} below) that the map 
\begin{equation}\label{makeperiodic}
F_{\mf a,\mf b}\colon z\mapsto \chi\big(\sigma_{\mf b}T^{-1}\sigma_{\mf b}^{-1}\big)^{\Rea z}E_{\mf a,\chi}\big(\sigma_{\mf b}.z, s\big),
\end{equation}
where 
\[
 T \sceq p_1 = \bmat{1}{1}{0}{1},
\]
is periodic in the real direction with period~$1$ (if $s$ is in the domain of convergence). We then use the Fourier expansion of $F_{\mf a,\mf b}$ and compose it termwise with 
\[
\chi\big(\sigma_{\mf b}T^{-1}\sigma_{\mf b}^{-1}\big)^{-\Rea z}
\]
in order to revert the periodization in \eqref{makeperiodic}. This results in a Fourier-type expansion of $E_{\mf a,\chi}(\sigma_{\mf b}.\cdot,s)$ and hence gives Fourier-type expansions of $E_{\chi}$ (one for each cusp).

In Theorem~\ref{thm:expansion} below we state these Fourier-type expansions. In the proof of Theorem~\ref{thm:expansion} we present a slightly more direct way than proceeding via periodization. The latter method however is lurking in the background, and we discuss it in a bit more detail in Remark~\ref{rem:periodic} below. 

For the statement of Theorem~\ref{thm:expansion} and its proof we need a few preparations.

Let $\mf a,\mf b\in\mf C$. We let $\mc R(\mf a,\mf b)$ denote the set of pairs $(c,d)\in\R_{>0}\times\R$ such that $d\in [0,c)$ and that there exists $a,b\in\R$ such that 
\[
 \bmat{a}{b}{c}{d}\in \sigma_{\mf a}^{-1}\Gamma\sigma_{\mf b}.
\]
For any $(c,d)\in\mc R(\mf a,\mf b)$ we fix an element 
\[
\omega_{(\mf a,\mf b)}^{(c,d)}\in\sigma_{\mf a}^{-1}\Gamma\sigma_{\mf b}
\]
of the form $\textbmat{*}{*}{c}{d}$. All results will be independent of the choice of the upper row of $\omega_{(\mf a,\mf b)}^{(c,d)}$. To simplify notation, we set
\[
 \omega_{(\mf a,\mf b)}^{-(c,d)}\sceq \left( \omega_{(\mf a,\mf b)}^{(c,d)}\right)^{-1}.
\]
We remark that this notation is safe of misinterpretations since $c$ is always chosen to be positive.

For any cusp representative $\mf c\in\mf C$ let 
\[
 g_{\mf c} \sceq \sigma_{\mf c}T\sigma_{\mf c}^{-1}.
\]
Let $r=r(\mf c)$ denote the number of Jordan chains of $\chi\big(g_{\mf c}^{-1}\big)$, let 
\[
n_j = n_{\mf c,j} \qquad (j = 1, \ldots, r),
\]
be the lengths of the Jordan chains (in any order, fixed once and for all), and fix a basis 
\[
\mc B(\mf c) =\big(e_{\mf c, 1}^{(1)},\ldots, e_{\mf c,n_1}^{(1)}, e_{\mf c, 1}^{(2)},\ldots, e_{\mf c, n_2}^{(2)},\ldots, e_{\mf c,1}^{(r)},\ldots, e_{\mf c,n_r}^{(r)})
\]
of $V$ consisting of Jordan chains of $\chi(g_{\mf c}^{-1})$. For $j\in\{1,\ldots, r\}$ let $\lambda_j=\lambda_{\mf c,j}$ be the eigenvalue of the $j$-th Jordan chain and let
\[
 J_j = 
 \begin{pmatrix}
  \lambda_j & 1 
  \\
  & \lambda_j & 1
  \\
  & & \ddots & \ddots 
  \\
  & & &  \lambda_j & 1
  \\
  & & & & \lambda_j
 \end{pmatrix}  \in \Mat(n_j\times n_j;\C)
\]
be the associated Jordan block. Then the Jordan normal form of $\chi(g_{\mf c}^{-1})$ with respect to the basis $\mc B(\mf c)$ is 
\begin{equation}\label{JNF}
\begin{pmatrix}
J_1 &
\\
& J_2 &
\\
& & \ddots 
\\
& & & J_{r}
\end{pmatrix}.
\end{equation}
Let 
\[
 N(\mf c) \sceq \max_{j=1}^{r(\mf c)} n_{\mf c,j}
\]
denote the maximal length of a Jordan chain.

For $z\in\C$ we set 
\[
 \nte(z)\sceq e^{2\pi i z}.
\]
Since $\chi$ has non-expanding cusp monodromy and $g_{\mf c}$ is parabolic, for each $j\in\{1,\ldots, r\}$, the eigenvalue $\lambda_j$ has absolute value $1$. Let $\mu_j=\mu_{\mf c,j}\in [0,1)$ be such that $\lambda_j = \nte(\mu_j)$.

Let $a\in\R$ and $k\in\N$. For each $p\in\{0,1,\ldots,k\}$ let $b_p(a,k)\in\R$ be the unique real number such that for all $u\in\R$ we have
\begin{equation}\label{stirling}
 \binom{u-a}{k} = \sum_{p=0}^k u^p b_p(a,k).
\end{equation}
Note that $b_p(a,k)$ is, up to scaling, a signed Stirling number of the first kind.

For $r\in\N_0$, $a,y\in\R$, $s\in\C$ with $\Rea s> (r+1)/2$  we set
\begin{equation}\label{def_PHI}
 \Phi(r,a,y,s) \sceq \int_{-\infty}^\infty \frac{x^r \nte(a x)}{ (x^2 + y^2)^s} dx.
\end{equation}
We note that this integral converges. Let $K_\nu$ denote the modified Bessel function of second kind. For details and properties we refer to \cite[Section~8]{Gradshteyn}. By \cite[3.251(2), 8.384(1)]{Gradshteyn} and Lemma~\ref{lem:julie} below we have
\begin{equation}\label{PHI}
\Phi(r,a,y,s) = 
\begin{cases}
y^{1+r-2s} \cfrac{\left( (-1)^r + 1\right) \Gamma\left(\frac{r+1}{2}\right) \Gamma\left(s-\frac{r+1}{2}\right) }{2\Gamma(s)} & \text{for $a=0$}
 \\[4mm]
\cfrac{2\pi^s}{i^r \Gamma(s)} y^{\frac12-s} \cfrac{\partial^r}{\partial a^r}\left( |a|^{s-\frac12} K_{s-\frac12}\big(2\pi|a|y\big)\right) & \text{for $a\not=0$.}
\end{cases}
\end{equation}
For any set $M$ and any elements $a,b\in M$ we set (Kronecker $\delta$-function)
\[
 \delta_{a,b} \sceq  
 \begin{cases}
  1 & \text{if $a=b$}
  \\
  0 & \text{otherwise.}
 \end{cases}
\]
Further, we let $\delta_{\Z}$ denote the indicator function of $\Z$, thus
\[
\delta_{\Z}(t) =
\begin{cases}
 1 & \text{if $t\in\Z$}
 \\
 0 & \text{otherwise.}
\end{cases}
\]

\begin{thm}\label{thm:expansion}
Let $c_1$ be as in Theorem~\ref{thm:estimate}, and let $\mf a,\mf b\in\mf C$. For 
\[
(z,s)\in \h\times\{s\in\C\mid\Rea s > 1+c_1\}
\]
we have
\begin{align*}
E_{\mf a,\chi}&(\sigma_{\mf b}.z,s)  = \delta_{\mf a,\mf b}y^s P_{\mf a} + \sum_{k=0}^{\lfloor (N(\mf b)-1)/2 \rfloor}\varphi_{\mf a,\mf b,2k+1}(s,x)y^{2k+1-s}  
 \\
 & \quad + y^{\frac12}\sum_{t\in\R\setminus\{0\}}\sum_{j=1}^{r(\mf b)} \sum_{k=0}^{n_{\mf b,j}-1} \sum_{n=0}^k  \psi_{\mf a,\mf b}(t,s,j,k,n,x) 
  \frac{\partial^n}{\partial a^n}\Big\vert_{a=t} \left( |a|^{s-\frac12} K_{s-\frac12}(2\pi |a| y) \right)
\end{align*}
where $z = x+iy$ with $x,y\in\R$, and
\begin{align*}
\varphi_{\mf a,\mf b,q}(s, x) &= \frac{\Gamma\left(\frac{q}{2}\right)\Gamma\left(s-\frac{q}{2}\right)}{\Gamma(s)}\sum_{j=1}^{r(\mf b)} \delta\big(\mu_{\mf b,j},0\big)
\\
& \qquad\qquad \times \sum_{k=q-1}^{n_{\mf b,j}-1}   \sum_{(c,d)\in\mc R(\mf a,\mf b)} c^{-2s}b_{q-1}\left( x +\tfrac{d}{c},k\right)Q_{\mf a,\mf b}\big( (c,d), 1, j\big),
\\
\psi_{\mf a,\mf b}(t,s,j,k,n,x) & = \frac{2 \pi^s}{i^n\Gamma(s)} \delta_\Z\big(\mu_{\mf b,j}+t\big) \nte\big( - k\mu_{\mf b,j}\big) 
\\
& \hphantom{\times} \times \sum_{(c,d)\in\mc R(\mf a,\mf b)} \nte\left(\left(x+\tfrac{d}{c}\right)t\right) b_n\left( x+\tfrac{d}{c},k\right) c^{-2s} Q_{\mf a,\mf b}\big( (c,d), k+1, j\big),
\\
Q_{\mf a,\mf b}\big( (c,d), m, j\big) &= \left\langle e_{\mf b,m}^{(j)}, \chi\big(\sigma_{\mf b}\omega_{(\mf a,\mf b)}^{-(c,d)}\sigma_{\mf a}^{-1}\big)P_{\mf a}\right\rangle e_{\mf b,m}^{(j)},
\end{align*}
where $\langle \cdot , \cdot \rangle$ denotes the inner product on $V$.
\end{thm}

\begin{remark}
The structure of the maps $\psi_{\mf a,\mf b}$ is one of the reasons why we call the expansion in Theorem~\ref{thm:expansion} to be of Fourier-type. The contribution $\nte\left(\left(x+\tfrac{d}{c}\right)t\right)$ is the standard part of a Fourier(-type) expansion, compare with \cite[Theorem~3.1.2, 3.1.3]{Venkov_book} or  \cite[Theorem~3.6]{Deitmar_Monheim_eisenstein}. It is here perturbed by the factor $b_n\left( x+\tfrac{d}{c},k\right)$. This factor only arises in those directions where the representation $\chi$ is not unitary in $\mf b$.

If $\chi$ is unitary in $\mf b$, then the expansion in Theorem~\ref{thm:expansion} is identical to the Fourier expansion developed in \cite[Theorem~3.6]{Deitmar_Monheim_eisenstein}. For unitary $\chi$ it is identical to the classical Fourier expansion of Eisenstein series with unitary twists in \cite[Theorem~3.1.2, 3.1.3]{Venkov_book}.
\end{remark}

\begin{proof}[Proof of Theorem~\ref{thm:expansion}]
Since the cusp representatives $\mf a,\mf b\in\mf C$ are considered to be fixed throughout this proof, we omit in the notation of several objects the dependence on $\mf a$ and $\mf b$ whenever misunderstandings are impossible. In particular, we set 
\[
 \mc R\sceq \mc R(\mf a,\mf b).
\]
Let 
\[
 \Omega_\infty\sceq \Lambda \sceq \bmat{1}{\Z}{0}{1} =\left\{ \bmat{1}{n}{0}{1}   \ \left\vert\ n\in\Z \vphantom{\bmat{1}{n}{0}{1} }\right.\right\}.
\]
For $(c,d)\in\mc R$ we set
\[
 \omega_{(c,d)}\sceq \omega_{(\mf a,\mf b)}^{(c,d)}
\]
and
\[
 \Omega_{(c,d)} \sceq \Lambda \omega_{(\mf a,\mf b)}^{(c,d)}\Lambda.
\]
By \cite[Theorem~2.7]{Iwaniec_book} we have the double coset decomposition 
\begin{equation}\label{doublecoset}
 \sigma_{\mf a}^{-1}\Gamma\sigma_{\mf b} = \delta_{\mf a,\mf b}\Omega_\infty \cup \bigcup_{(c,d)\in\mc R} \Omega_{(c,d)}.
\end{equation}
Note that even though the statement of \eqref{doublecoset} in \cite{Iwaniec_book} is restricted to nonuniform cofinite Fuchsian groups, it and its proof in \cite{Iwaniec_book} apply---without any changes---to non-cofinite Fuchsian groups with cusps as well.

One easily sees that 
\[
 \Gamma_{\mf a}\backslash\Gamma = \Lambda\backslash \sigma_{\mf a}^{-1}\Gamma\sigma_{\mf b}.
\]
Thus,  
\begin{align*}
 E_{\mf a,\chi}(\sigma_{\mf b}.z,s) & = \sum_{[g]\in\Gamma_{\mf a}\backslash\Gamma} \big(\Ima \sigma_{\mf a}^{-1}g\sigma_{\mf b}.z\big)^s \chi(g^{-1})P_{\mf a}
 \\
 & = \sum_{[h]\in \Lambda\backslash \sigma_{\mf a}^{-1}\Gamma \sigma_{\mf b}} \big( \Ima h.z \big)^s \chi\big( \sigma_{\mf b}h^{-1}\sigma_{\mf a}^{-1}\big) P_{\mf a}
 \\
 & = \delta_{\mf a,\mf b} y^s P_{\mf a} + \sum_{(c,d)\in\mc R} \sum_{[h]\in\Lambda\backslash \Omega_{(c,d)}} \big(\Ima h.z\big)^s \chi\big(\sigma_{\mf b} h^{-1} \sigma_{\mf a}^{-1}\big) P_{\mf a}
 \\
 & = \delta_{\mf a,\mf b} y^s P_{\mf a} 
 \\
 & \ \quad+ \sum_{(c,d)\in\mc R} \sum_{m\in\Z} \big(\Ima \omega_{(c,d)}.(z+m)\big)^s \chi\big(g_{\mf b}^{-1}\big)^m \chi\big(\sigma_{\mf b}\omega_{(c,d)}^{-1}\sigma_{\mf a}^{-1}\big)P_{\mf a}.
\end{align*}
In order to investigate the latter sum in more detail and to separate the contributions of the real part $x$ and the imaginary part $y$ of $z$, we express its summands with respect to the basis $\mc B(\mf b)$. With respect to this basis, the representing matrix of $\chi\big(g_{\mf b}^{-1}\big)$ is in Jordan normal form, see \eqref{JNF}. In these coordinates we find
\begin{align}\label{Eis_coord}
E_{\mf a,\chi}(\sigma_{\mf b}.z,s) & = \delta_{\mf a,\mf b} y^s P_{\mf a} 
+ \sum_{(c,d)\in\mc R} \sum_{j=1}^{r(\mf b)} \sum_{k=0}^{n_{\mf b,j}-1} \nu_{\mf b}\big((c,d),z,k,j\big) Q_{\mf a,\mf b}\big( (c,d), k+1, j\big)
\end{align}
where
\begin{equation}\label{nu}
 \nu_{\mf b}\big((c,d),z,k,j\big) \sceq \sum_{m\in\Z}\big(\Ima \omega_{(c,d)}.(z+m)\big)^s \binom{m}{k} \lambda_{\mf b,j}^{m-k},
\end{equation}
\[
 Q_{\mf a,\mf b}\big( (c,d), k+1, j\big) \sceq \left\langle e_{\mf b,k+1}^{(j)}, \chi\big(\sigma_{\mf b}\omega_{(c,d)}^{-1}\sigma_{\mf a}^{-1}\big)P_{\mf a}\right\rangle e_{\mf b,k+1}^{(j)}.
\]
In what follows we deduce another expression for $\nu_{\mf b}$ which allows us to separate the contributions of $x$ and $y$. 

Recall that $\mu_{\mf b,j}\in [0,1)$ is chosen such that $\lambda_{\mf b,j} = \nte\big( \mu_{\mf b,j} \big)$. The Poisson summation formula implies
\begin{align*}
 \nu_{\mf b}\big((c,d),z,k,j\big) & = \sum_{m\in\Z}\big(\Ima \omega_{(c,d)}.(z+m)\big)^s \binom{m}{k} \nte\big((m-k)\mu_{\mf b,j}\big)
 \\
 & = \sum_{\ell\in\Z} \int_\R \big(\Ima \omega_{(c,d)}.(z+t)\big)^s \binom{t}{k} \nte\big((t-k)\mu_{\mf b,j}\big)e(-t\ell)  \, dt
 \\
 & = \sum_{\ell\in\Z}\nte\big(-k\mu_{\mf b,j}\big) \nte\left( \left(x+\tfrac{d}{c}\right)(\ell-\mu_{\mf b,j})\right) 
 \\
 & \qquad \times \int_\R \binom{u-x-\frac{d}{c}}{k} \left(\frac{y}{c^2(u^2+y^2)}\right)^s \nte\big( u(\mu_{\mf b,j}-\ell) \big)\, du,
\end{align*}
where, for the last equality, we used the substitution 
\[
u(t) = t+x+\frac{d}{c}.
\]
Using \eqref{stirling} and \eqref{def_PHI} we get
\begin{align*}
\nu_{\mf b}\big((c,d),z,k,j\big) & = \sum_{\ell\in\Z}\nte\big(-k\mu_{\mf b,j}\big) \nte\left( \left(x+\tfrac{d}{c}\right)(\ell-\mu_{\mf b,j})\right) 
 \frac{y^s}{c^{2s}} 
 \\
  & \qquad\qquad\qquad \times \sum_{p=0}^k b_p\big( x + d/c, k\big) \int_\R \frac{u^p \nte\left( u\big(\mu_{\mf b,j} - \ell\big) \right) }{ \big( u^2 + y^2 \big)^s} \, du
\\
& = \sum_{\ell\in\Z}\nte\big(-k\mu_{\mf b,j}\big) \nte\left( \left(x+\tfrac{d}{c}\right)(\ell-\mu_{\mf b,j})\right) \frac{y^s}{c^{2s}} 
\\
  & \qquad\qquad\qquad \times
\sum_{p=0}^k b_p\big( x + d/c, k\big) \Phi\big(p,\mu_{\mf b,j}-\ell, y,s\big).
\end{align*}
Recall that the possible choices for the constant $c_1$ in the statement of Theorem~\ref{thm:estimate} are governed by the restrictions implied by Proposition~\ref{prop:hyp}, and hence by the restrictions given by \cite[Corollary~3.4]{FP_szf}. These restrictions yield that $c_1> N(\mf b)-1$, and hence $\Rea s > (p+1)/2$ for all values of $s$ that we consider here. In turn, $\Phi(p,\cdot,\cdot, s)$ is indeed well-defined. 
Applying \eqref{PHI} we find 
\begin{align*}
\nu_{\mf b}&\big((c,d),z,k,j\big)
= \sum_{\ell\in\Z}\nte\big(-k\mu_{\mf b,j}\big) \nte\left( \left(x+\tfrac{d}{c}\right)(\ell-\mu_{\mf b,j})\right) c^{-2s}
\\
& \quad\times \Bigg[ \delta\big(\mu_{\mf b,j},\ell\big) \sum_{p=0}^k y^{p+1-s} b_p(x+d/c,k) \frac{(-1)^p + 1}{2} \frac{\Gamma\left(\frac{p+1}{2}\right)\Gamma\left(s-\frac{p+1}{2}\right) }{\Gamma(s)}
\\
& \quad \qquad + \left(1-\delta\big(\mu_{\mf b,j},\ell\big)\right) y^{\frac12} \frac{2\pi^s}{\Gamma(s)} \sum_{n=0}^k i^{-n} b_n(x+d/c,k)
\\
& \hphantom{\left(1-\delta\big(\mu_{\mf b,j},\ell\big)\right) y^{\frac12} \frac{2\pi^s}{\Gamma(s)} \sum_{n=0}^k i^{-n}} 
\times \frac{\partial^n}{\partial a^n}\Big\vert_{a=\mu_{\mf b,j}-\ell} \left( |a|^{s-\frac12} K_{s-\frac12}\big(2\pi|a|y\big)\right) \Bigg].
\end{align*}
 
Using this expression for $\nu_{\mf b}$ in \eqref{Eis_coord} and reordering sums completes the proof.
\end{proof}

In the introduction to this section we sketched the possibility to prove Theorem~\ref{thm:expansion} using a periodization of the Eisenstein series at cusps.  In what follows we provide a few more details of this approach. In a forthcoming publication we will use this approach to develop Fourier-type expansion in a much more general setup.

One key ingredient is the following result on the transformation behavior of twisted Eisenstein series. It is a generalization of \cite[Theorem 3.1.1(1)]{Venkov_book}.

\begin{lemma}\label{Eisentrans}
Let $c_1$ be as in Theorem~\ref{thm:estimate}, $\mf c\in\mf C$ and $\gamma\in\Gamma$. For all $(z,s)\in\h\times\C$ with $\Rea s>1+c_1$ we have
\[
 E_{\mf c,\chi}(\gamma.z,s) = \chi(\gamma)E_{\mf c,\chi}(z,s).
\]
\end{lemma}

\begin{proof} We have
\begin{align*}
 E_{\mf c,\chi}(\gamma.z,s) & = \sum_{[g]\in\Gamma_{\mf c}\backslash\Gamma} \left( \Ima\sigma_{\mf c}^{-1} g\gamma.z\right)^s  \chi\big(g^{-1}\big)\circ P_{\mf c}
 \\
 & = \sum_{[h]\in\Gamma_{\mf c}\backslash\Gamma} \left(\Ima\sigma_{\mf c}^{-1}h.z\right)^s \chi\big(\gamma h^{-1}\big)\circ P_{\mf c}
 \\
 & = \sum_{[h]\in\Gamma_{\mf c}\backslash\Gamma} \left(\Ima\sigma_{\mf c}^{-1}h.z\right)^s \chi(\gamma)\chi\big(h^{-1}\big)\circ P_{\mf c}
 \\
 & = \chi(\gamma) E_{\mf c,\chi}(z,s).
\end{align*}
This completes the proof.
\end{proof}

Finally we provide the periodization.

\begin{remark}\label{rem:periodic}
Let the notation be as for Theorem~\ref{thm:expansion}, and let 
\[
 B_{\mf b}\colon \R \to \End(V),\quad x\mapsto \chi(g_{\mf b})^{-x}
\]
be the (unique) group homomorphism which extends continuously the group homomorphism 
\[
 \Z\to \End(V),\quad n\mapsto \chi(g_{\mf b})^{-n}.
\]
Define
\[
 F(z) \sceq B_{\mf b}(x)E_{\mf a,\chi}(\sigma_{\mf b}.z,s).
\]
We have
\begin{align*}
E_{\mf a,\chi}(\sigma_{\mf b}.(z+1), s) & = E_{\mf a,\chi}(\sigma_{\mf b}T.z, s)
\\
& = E_{\mf a,\chi}(\sigma_{\mf b}T\sigma_{\mf b}^{-1}\sigma_{\mf b}.z, s)
\\
& = E_{\mf a,\chi}(g_{\mf b}\sigma_{\mf b}.z, s)
\\
& = \chi(g_{\mf b}) E_{\mf a,\chi}(\sigma_{\mf b}.z, s).
\end{align*}
Therefore, 
\begin{align*}
 F(z+1) & = B_{\mf b}(x+1) E_{\mf a,\chi}(\sigma_{\mf b}.(z+1),s)
 \\
 & = \chi(g_{\mf b})^{-(x+1)}\chi(g_{\mf b}) E_{\mf a,\chi}(\sigma_{\mf b}.z, s)
 \\
 & = B_{\mf b}(x) E_{\mf a,\chi}(\sigma_{\mf b}.z, s)
 \\
 & = F(z).
\end{align*}
Thus, $F$ is $1$-periodic. Now we can proceed as in the proof of Theorem~\ref{thm:expansion} to develop the Fourier expansion of $F$. If we multiply this Fourier expansion with 
\[
B_{\mf b}(-x) = B_{\mf b}^{-1}(x)
\]
then we find again the expansions of $E_{\mf a,\chi}(\sigma_{\mf b}.z,s)$ as stated in 
Theorem~\ref{thm:expansion}.
\end{remark}

\appendix

\section{}\label{Julie_lemma}

In the proof of Theorem~\ref{thm:expansion} the integral  
\begin{equation}\label{Phi_def}
 \Phi(r,a,y,s) = \int_{-\infty}^\infty \frac{x^r \nte(a x)}{ (x^2 + y^2)^s} dx
\end{equation}
arose, which we substituted with the expressions \eqref{PHI}. For $a=0$, the equivalence of \eqref{Phi_def} and \eqref{PHI} is given by \cite[3.251(2), 8.384(1)]{Gradshteyn}. For the remaining cases, the equivalence was kindly shown to us by Julie Rowlett. In the following we provide her proof.

\begin{lemma}[J.\@ Rowlett]\label{lem:julie}
For any $r\in\N_0$, $a\in\R$, $a\not=0$, $y\in\R$ and $s\in\C$ with $\Rea s> (r+1)/2$ we have
\[
 \Phi(r,a,y,s) = 
\cfrac{2\pi^s}{i^r \Gamma(s)} y^{\frac12-s} \cfrac{\partial^r}{\partial a^r}\left( |a|^{s-\frac12} K_{s-\frac12}\big(2\pi|a|y\big)\right).
\]
\end{lemma}

\begin{proof}
Throughout let 
\[
(\mc F f)(\xi) \sceq \int_\R f(x)e^{-2 \pi i\xi x}dx
\]
and
\[
(\mc F^{-1} f)(x) = \int_\R f(\xi) e^{2 \pi i\xi x}d\xi
\]
denote the Fourier transform and its inverse, respectively. Note that 
\[
 \Phi(0,a,y,s) = \mc F^{-1} \left[ \frac{1}{ (x^2 + y^2)^s} \right](a).
\]
By \cite[p. 172]{Watson44} (note that $a\not=0$), we have
\[
\Phi(0,a,y,s) = 
\frac{2\pi^s|a|^{s-\frac12}}{\Gamma(s)} y^{\frac12-s} K_{s-\frac12}(2\pi |a| y).
\]
Using standard properties of the Fourier transformation it follows that for all $r\in\N_0$ we have
\begin{align*}
\Phi(r,a,y,s) &= \int_{-\infty}^\infty \frac{x^r \nte(a x)}{ (x^2 + y^2)^s} dx 
\\
& = \mc F^{-1} \left[ \frac{x^r}{ (x^2 + y^2)^s} \right] (a)
\\
& =  i^{-r} \frac{\partial^r}{\partial a^r} \mc F^{-1} \left[ \frac{1}{ (x^2 + y^2)^s} \right] (a) 
\\
& = i^{-r} \frac{\partial^r}{\partial a^r} \left(  \frac{2\pi^s|a|^{s-\frac12}}{\Gamma(s)} y^{\frac12-s} K_{s-\frac12}(2\pi |a| y) \right)
\\
& = \frac{2\pi^s}{i^r \Gamma(s)} y^{\frac12-s} \frac{\partial^r}{\partial a^r}\left( |a|^{s-\frac12} K_{s-\frac12}\big(2\pi|a|y\big)\right). \qedhere
\end{align*}
\end{proof}

\bibliography{eisenstein} 
\bibliographystyle{amsplain}

\end{document}